\documentclass{amsart}

\usepackage{amsmath} \usepackage{amsfonts} \usepackage{amssymb} 

\newcommand{\z}{\mathbb Z} 
\newcommand{\n}{\mathbb N} 
\newcommand{\q}{\mathbb Q} 
\newcommand{\f}{\mathbb F} 

\newcommand{\zp}{{\mathbb Z}_p} 
\newcommand{\qp}{{\mathbb Q}_p} 
 
\newcommand{\op}{{\mathcal O}}
 
\newcommand{\kp}{\kappa} 
\newcommand{\ip}{{\mathfrak m}}
\newcommand{\fp}{{{\mathbb F}_p}}

\newcommand{\ipalg}{\ip^\mathrm{alg}}

\newcommand{\Nott}{\mathrm{Nott}}
\newcommand{\Aut}{\mathrm{Aut}} 
\newcommand{\Gal}{\mathrm{Gal}} 
\newcommand{\End}{\mathrm{End}}

\newcommand{\Norm}{\mathrm{Norm}}
\newcommand{\Normsep}{\mathrm{Norm}^\mathrm{sep}}
\newcommand{\wideg}{\mathrm{ord}_x}

\newcommand{\snc}{\mathcal{S}_0} 
\newcommand{\gnc}{\mathcal{G}_0}

\newcommand{\suchthat}{\; | \;}

\newcommand{\G}{\mathfrak G}
\newcommand{\ug}{\mathfrak u}
\newcommand{\zg}{\mathfrak w}
\newcommand{\anyg}{\mathfrak g}
\newcommand{\anyh}{\mathfrak h}

\newcommand{\cala}{\mathcal A}

\newcommand{\Newt}{\mathcal N}

\newcommand{\E}{\mathcal E}

\newcommand{\Llf}{\Lambda_f}
\newcommand{\Llu}{\Lambda_u}
\newcommand{\Oof}{\Omega_f}

\newtheorem{thm}{Theorem}[section] 
\newtheorem{lemma}[thm]{Lemma}     

\newtheorem{cor}[thm]{Corollary} \newtheorem{prop}[thm]{Proposition} 
\newtheorem{mainthm}[thm]{Main Theorem}     
\newtheorem{remark}[thm]{Remark}

\title{Galois extensions of height-one commuting dynamical systems} 
\author{Ghassan Sarkis}
\address{Pomona College\\610 N. College Ave.\\Claremont, CA 91711}
\email{ghassan.sarkis@pomona.edu}
\author{Joel Specter}
\address{Wesleyan University\\265 Church St.\\Middletown, CT 06459-0128}
\email{jspecter@wesleyan.edu}
\thanks{The authors were partially supported by NSF Grant DMS-0755540.}

\begin{document} 

\begin{abstract} We consider a dynamical system consisting of a pair of commuting power series, one noninvertible and another nontorsion invertible, of
height one with coefficients in the $p$-adic integers. Assuming that each point of the dynamical system generates a Galois extension over the base field, we show that these extensions are in fact abelian, and, using results and considerations from the theory of the field of norms, we also show that the dynamical system must include a torsion series of maximal order. From an earlier result, this shows that the series must in
fact be endomorphisms of some height-one formal group. \end{abstract}

 \maketitle




\section{Introduction} The study of $p$-adic dynamical systems has seen increased interest over the past two decades,
reflected most recently in a new MSC category: {\sl Arithmetic and non-Archimedean dynamical systems}. This note is
concerned with three overlapping ways of looking at such systems---formal power series that commute under composition, iterated morphisms of
the open $p$-adic unit disc, and galoisness of extensions that are obtained by adjoining zeros of dynamical systems. 
Indeed, the proof of the main result in this note can be viewed as relating commuting power series to formal groups, analytic maps of the open unit disk to locally analytic galois automorphisms, and galois towers to automorphism subgroups of residue fields.

\subsection{Notation and motivation} Our power series have no constant term in order for composition to be well defined and finitary. Also, their linear coefficients are nonzero to exclude trivial cases. We adopt therefore some of the notation of \cite{Lubin:Nads}.  Let $g^{\circ n}$ be the $n$-fold iterate of $g$ with itself under composition. For a commutative ring $R$, let $\snc(R)=\{g\in R[\![x]\!]\suchthat g(0)=0\;\mathrm{and}\; g'(0)\neq 0\}$.
Let $\gnc(R)=\{g\in\snc(R)\suchthat g'(0)\in R^\times\}$ be the group of series that are invertible under composition.

Suppose $F$ is a finite extension of $\qp$ with ring of integers $\op$, maximal ideal $\ip$, and residue field $\kp=\op/\ip$.  Denote by $v_F$ the unique additive valuation on any algebraic extension of $F$ normalized so that $v_F(F^\times)=\z$; for simplicity, $v_p$ will be used instead of $v_{\qp}$.  Let $\ipalg$ be the maximal ideal in the integral closure of $\op$ in $F^\mathrm{alg}$, an algebraic closure of $F$.

The {\sl Newton Polygon} of $g(x)=\sum a_i x^i\in\op[\![x]\!]$, denoted $\Newt(g)$, is the convex hull of the sequence of points $(i,v_F(a_i))$. If $\Newt(g)$ has a segment of horizontal length $\ell$ and slope $\lambda$, then $g$ has, counting multiplicity, precisely $\ell$ roots in $F^\mathrm{alg}$ of $F$-valuation $\lambda$.  We are interested  in roots that lie in $\ipalg$; these correspond to segments of the Newton polygon of negative slope.  To that end, we define $\Newt^-(g)$ to be the portion of $\Newt(g)$ consisting of segments whose slopes are negative.

Let $\bar{g}(x)=\sum\bar{a}_ix^i\in\kp[\![x]\!]$ be the coefficientwise reduction of $g$ to $\kp$. The {\sl Weierstrass degree} of $\bar{g}$, denoted $\wideg(\bar{g})$, is defined to equal $\infty$ if $\bar{g}=0$, and $\min\{i\suchthat \bar{a}_i\neq 0\}$ otherwise; the Weierstrass degree of $g$ is defined to equal $\wideg(\bar{g})$.
The {\sl $p$-adic Weierstrass Preparation Theorem} (WPT) asserts that if $\wideg(g)<\infty$ then there exists a unique factorization $g(x)=P(x)U(x)$, where $P(x)\in\op[x]$ is monic of degree $\wideg(g)$ and $U(x)\in\op[\![x]\!]$ has a multiplicative inverse, and hence no zeroes in $\ipalg$. Moreover, the roots of $P(x)$ and $g(x)$ in $\ipalg$ coincide; consequently, so do $\Newt(P)$ and $\Newt^-(g)$. See \cite[Chapter IV]{Koblitz} for a more details on Newton polygons.

If $f,u\in\snc(\op)$ such that $f$ is noninvertible and $u$ is invertible, let 
$$\begin{array}{rclcrcl}
\Llf(n)&=&\{\pi\in\ipalg\suchthat f^{\circ n}(\pi)=0\}&\quad\mathrm{and}\quad& \Llf&=&\cup_{n\geq 0}\Llf(n);\\
 \Llu(n)&=&\{\pi\in\ipalg\suchthat u^{\circ p^n}(\pi)=\pi\} &\quad\mathrm{and}\quad& \Llu&=&\cup_{n\geq 0}\Llu(n).
 \end{array}$$
Observe that $n\geq 0$, $\Llf(n+1)\backslash\Llf(n)$ consists of roots of $f(x)-\pi$ as $\pi$ ranges through $\Llf(n)\backslash\Llf(n-1)$.  Let $\Oof(n)=\Llf(n)\backslash\Llf(n-1)$.
 
Although formal groups are not explicitly visible in our results, they provide part of the motivation, which we discuss briefly next. We will call $f,u\in\snc(\op)$ a {\sl commuting pair} if $f$ is noninvertible and $u$ is nontorsion invertible.
Commuting pairs share certain characteristic properties with formal group endomorphisms. For example,  $\Llf=\Llu$ by \cite[Proposition 3.2]{Lubin:Nads}.  Also, $\wideg(f)$ is either infinite or a power of $p$ by \cite[Main Theorem 6.3]{Lubin:Nads}. Both of these results are important properties of formal group endomorphisms.
Thus, Lubin suggested that commutativity may be enough to indicate the existence of ``a formal group somehow in the background'' \cite[page 341]{Lubin:Nads}.   Counterexamples to na\"ive statements and proofs of special cases of this conjecture are both known, though a general case remains elusive.  
For a more detailed discussion of the issues involved in a precise statement of the conjecture, see \cite{Sarkis:HeightOne}.

Let $e=p-1$ if $p>2$ and $e=2$ if $p=2$.  The following special case of Lubin's conjecture, proven in \cite[Theorem 1.1]{Sarkis:HeightOne}, makes use of a torsion third series of order $e$ in the dynamical system: 

\begin{thm} Let $f,u,z\in\snc(\zp)$ such that $f,u$ is a commuting pair with $\wideg(f)=p$ and $v_p(f'(0))=v_p(u'(0)-1)=1$, and if $p=2$ then additionally $v_2(u'(0)^2-1)=3$. Suppose also that $z$ is torsion of order $e$ and commutes with $f$.  Then there exists a formal group $F$ over $\zp$ such that $f,u,z\in\End_{\zp}(F)$.
\end{thm}

It is a straightforward corollary to this theorem that $\qp(\pi)/\qp$ is Galois for all $\pi\in\Llf$. We will show that, conversely, the Galoisness of $\qp(\pi)/\qp$ is sufficient to guarantee the existence of the torsion series. 

Let $\delta=1$ if $p>2$ and $\delta=2$ if $p=2$.  We will call a commuting pair $f,u\in\snc(\zp)$ {\sl minimal} when 
\begin{itemize}
\item $\wideg(f)=p$;
\item $v_p(f'(0))=1$; and
\item $v_p(u'(0)-1)=\delta$.
\end{itemize}
Note that the condition on $v_p(u'(0)-1)$ when $p=2$ is slightly different than the one in \cite{Sarkis:HeightOne}, though the two are equivalent in the contexts we consider. 
Our main result is the following:

\begin{mainthm}\label{thm:main} Suppose $f,u\in\snc(\zp)$ is a minimal commuting pair.  If $\qp(\pi)/\qp$ is Galois for each $\pi\in\Llf$, then there exists a torsion series $z\in\snc(\zp)$ of order $e$ commuting with $f$ and $u$.
\end{mainthm}

\begin{remark}\label{remark:proofplan}{\rm By \cite[Propositions 1.1 and 1.2]{Lubin:Nads}, there exists a unique power series $L_f(x)\in\qp[\![x]\!]$ for which $L_f'(0)=1$ and $L_f\circ f=f'(0)L_f$. And for each $a\in\qp$ there exists a unique $[a]_f\in\qp[\![x]\!]$ such that $[a]_f'(0)=a$ and $[a]_f\circ f=f\circ [a]_f$.  Let $\zeta_e$ be a primitive $e^\mathrm{th}$ root of unity, and let $z=[\zeta_e]_f\in\qp[\![x]\!]$.  In order to prove our main result, we need show only that  $z\in\zp[\![x]\!]$.
}\end{remark}


\section{Roots and fixed points}


The integrality of the torsion series $z$ will rely on the existence of a torsion series of the same order in $\gnc(\fp)$ that commutes with $\bar f$ and $\bar u$.  More generally, it will rely on information about the structure of the normalizer of $\bar u$ in $\gnc(\fp)$ that follows from the theory of the fields of norms.
We begin this section with a brief overview of fields of norms, culminating in \cite[Th\'eor\`eme 5.9]{LMS}.  

Suppose $K$ is a local field with residue field $\kp$ of characteristic $p>0$, $L/K$ is an infinite totally ramified abelian extension, and $\E_{L/K}=\{E\suchthat K\subseteq E\subset L\;\mathrm{and}\;[E:K]<\infty\}$.  Let $\displaystyle{X_K(L)^*=\lim_{\longleftarrow\atop E\in\E_{L/K}} E^*}$, where the transition morphisms are the norm maps, and let $X_K(L)=X_K(L)^*\cup\{0\}$.  An element $\alpha$ of $X_K(L)$ is therefore an indexed family $(\alpha_E)_{E\in\E_{L/K}}$ such that whenever $E'\supseteq E$ then $N_{E'/E}(\alpha_{E'})=\alpha_E$.  Remarkably, the set of norms $X_K(L)$ turns out to be isomorphic to the local field $\kp((x))$. Multiplication is defined in an obvious way in light of the multiplicativity of the norm map: if $\alpha,\beta\in X_L(K)$ then $\alpha\beta=(\alpha_E\beta_E)E\in\E_{L/K}$. Addition is not as obvious: for $E\in\E_{L/K}$, the norms $N_{E'/E}(\alpha_{E'}+\beta_{E'})$ converge over $E'\in\E_{L/E}$ to an element $\gamma_E\in E$, using which one can define $\alpha+\beta=(\gamma_E)_{E\in\E_{L/K}}$.  Finally, $v_{X_{L/K}}(\alpha)$ can be defined as $v_E(\alpha_E)$, which is independent of $E$.  The image of $\Gal(L/K)$ under $X_K(*)$ is a closed subgroup of the automorphism group of $\kp((x))$.


Conversely, certain subgroups of the automorphism group of $\kp((x))$ arise as images of Galois groups under the field-of-norms functor.  The {\sl Nottingham group} $\Nott(\kp)=\{g(x)\in\snc(\kp)\suchthat g'(0)=1\}$ has elicited interest among group theorists because every countably-based pro-$p$ groups can be embedded in it, and as such, it contains elements of order $p^n$ for all $n$; the shape of such torsion elements will be important for the proof of our main result when $p=2$.   The Nottingham group is also the group of normalized automorphisms of the field $\kp((x))$. Some of its subgroups, like the one generated by $\bar u$, arise as images of Galois groups under the field-of-norms functor.  See \cite{Camina,Klopsch} for more information about the Nottingham group, \cite{FW1,FW2} for the original construction of the field of norms,  and \cite{LMS} for applications of the field of norms to $p$-adic dynamical systems that we make use of next.


If $\omega\in\Nott(\kp)$, let $i_n(\omega)=\wideg(\omega^{\circ p^n}(x)-x)-1$. This sequence of integers, called the {\sl lower ramification numbers} of $\omega$, measures the rapidity with which $\omega^{\circ p^n}$ approaches the identity.  According to Sen's Theorem \cite{Sen}, if $\omega^{\circ p^n}(x)\neq x$ then $i_n(\omega)\equiv i_{n-1}(\omega)$ mod $p^n$.
Let $e(\omega)=\lim_{n\to\infty}(p-1)i_n/p^{n+1}$. In light of Sen's Theorem, $e(\omega)$ is finite in those cases when $\omega^{\circ p^n}$ approaches the identity as slowly as possible.  Note that the factor $p-1$ normalizes $e(\omega)$ so that, when finite, it is an integer.

Let $\cala_\omega=\{\omega^{\circ a}\suchthat a\in\zp\}$ be the closed subgroup of $\snc$ generated by $\omega$.  The {\sl separable normalizer} of $\cala_\omega$ is given by $\Normsep_\kp(\cala_\omega)=\{\vartheta\in\kp[\![x]\!]\suchthat \vartheta'\neq 0\;\mathrm{and}\;\vartheta\circ \omega\circ\vartheta^{\circ -1}\in\cala\}$. By \cite[Proposition 5.5]{LMS}, $\Normsep_\kp(\cala_\omega)$ is in fact a group, and \cite[Th\'eor\`eme 5.9]{LMS}, quoted next, characterizes this group in the case that $\omega$ approaches the identity slowly.

\begin{thm}\label{thm:LMS}
Let $\kp$ be a perfect field and $\omega\in\snc(\kp)$ such that $e(\omega)<\infty$. Then $\Normsep_\kp(\cala_\omega)$ is an extension of a finite group of order dividing $e(\omega)$ by the group $\cala_\omega$.
\end{thm}

\begin{proof}[Proof Outline] The field-of-norms functor associates to $\cala_\omega$ the group $\Gal(L/K)$, where $L/K$ is an infinite abelian totally ramified extension of local fields with residue field $\kp$.  As such, the lower ramification numbers of $\bar{u}$ correspond to the lower ramification numbers of $\Gal(L/K)$, and $e(\omega)$ corresponds to the absolute ramification index of $K$, namely, $[K:K_0]$, where $K_0$ is the field of fractions of the Witt ring over $\kp$.  Let $\Aut(L/K)$ be the group of automorphisms of $L$ that send $K$ to itself and that act trivially on the residue field. The restriction map $\Aut(L/K)\to\Aut(K)$ has kernel $\Gal(L/K)$.  The order of its image divides $[K:K_0]$. \end{proof}

\subsection{The lower ramification numbers of $\bar u$} For the remainder of his section, we will consider the roots of $f$ and its iterates in order to show that $e(\bar u)=e=\begin{cases}p-1&\mathrm{if}\;p>2\\2&\mathrm{if}\;p=2\end{cases}$, and hence apply Theorem \ref{thm:LMS} to $\cala_{\bar u}$.  

Continue to denote by $F$ a finite extension of $\qp$ with ring of integers $\op$, maximal ideal $\ip$, and residue field $\kp$.


\begin{lemma}\label{lemma:totram}
Suppose $f\in\snc(\op)$ such that $\wideg(f)=p$ and $v_F(f'(0))=1$.  Then the roots of $f^{\circ n}$ in $\ipalg$ are simple for all $n$. Moreover, if $\pi\in\Oof(n)$ then  $F(\pi)/F$ is a totally ramified extension of degree $(p-1)p^{n-1}$, and $\pi$ is a uniformizer in $F(\pi)$.
\end{lemma}	

\begin{proof}  Using WPT, write $f(x)=P_0(x)U_0(x)$. Note that $\Oof(1)$ consists of the roots of $P_0(x)/x$.  By the hypothesis on $f$, $\Newt(P_0(x)/x)$ consists of a single segment from $(0,1)$ to $(p-1,0)$. So $P_0(x)/x$ is a degree $p-1$ Eisenstein polynomial over $\op$. Therefore, the roots of $P_0(x)/x$ are simple, each with $F$-valuation $1/(p-1)$, and each generating a degree $p-1$ totally ramified extension of $F$.

Proceeding by induction, assume that the result holds for some $n\geq 0$.  Let $\pi\in\Oof(n)$.  Using WPT again, write $f(x)-\pi=P_n(x)U_n(x)$, where $P_n(x)\in(\op[\pi])[x]$ is a polynomial whose Newton polygon consists of a single segment from $(0,1)$ to $(p,0)$; that is, $P_n(x)$ is a degree $p$ Eisenstein polynomial over $\op[\pi]$.  Thus, the roots of $f(x)-\pi$ are simple, each of $F(\pi)$-valuation $1/p$, and each generating a degree $p$ totally ramified extension of $F(\pi)$. Finally, if $\pi'\in\Oof(n)$ with $\pi'\neq\pi$, then the roots of $f(x)-\pi$ and $f(x)-\pi'$ are distinct.  \end{proof} 

\begin{lemma}\label{lemma:fandzeta}
Suppose $f\in\snc(\op)$ such that $\wideg(f)=p$ and $v_F(f'(0))=1$.  Let $\pi\in\Oof(1)$. Then $F(\pi)/F$ is Galois. Also, for a fixed primitive $p-1$ root of unity $\zeta_{p-1}$, and for each $0\leq i\leq p-2$, there exists a unique $\pi^{(i)}\in\Oof(1)$ such that $\pi^{(i)}\equiv\zeta_{p-1}^i\pi$ mod $\ip^2$.
\end{lemma}

\begin{proof} Using WPT as in Lemma \ref{lemma:totram}, write $f=P_0U_0$ and note that $v_p(P'_0(0))=1$ and $P_0(x)\equiv x^p$ mod $p$.  Therefore, $P_0$ is an endomorphism of a height-one formal $\zp$-module (see \cite{LT1}).  Since the roots of $f$ and $P_0$ coincide, the result follows.\end{proof}


\begin{cor}\label{cor:G1cyclic}
$\Gal(F(\pi)/F)$ is cyclic of order $p-1$.
\end{cor}

\begin{proof}
The Galois group is generated by $\pi\mapsto\pi^{(1)}$.
\end{proof}

We next quote \cite[Lemma 1.2]{Sarkis:HeightOne} for reference.

\begin{lemma}\label{lemma:samenewton}
Suppose $g_1,g_2 \in\snc(\op)$ such that $0<v_F(g_1'(0))=v_F(g_2'(0))<\infty$, and every root of $g_1$ in $\ip^\mathrm{alg}$ is also a root of $g_2$ of at least the same multiplicity.  Suppose further that $g_1\notin \snc(\ip)$. Then $\Newt^-(g_1)=\Newt^-(g_2)$, and so the roots of $g_1$ and $g_2$ in $\ip^\mathrm{alg}$ coincide.
\end{lemma}

\begin{lemma}\label{lemma:basecase}
Suppose $f,u\in\snc(\op)$ is a commuting pair such that $v_p(f'(0))=v_p(u'(0)-1)=1$ and $\wideg(f)=p$. Then $\Llf(1)=\Llu(0)$.
\end{lemma}

\begin{proof} If $\pi$ is a nonzero root of $f$, then $u(f(\pi))=f(u(\pi))=0$.  Thus $u(\pi)$ is another nonzero root of $f$, and by the hypothesis on $u'(0)$, it is in fact of the form $u(\pi)=\pi+\pi^2d$ for some $d\in\zp[\![\pi]\!]$.  By Lemma \ref{lemma:fandzeta}, $u(\pi)=\pi^{(0)}=\pi$. So every root of $f$ is also a root of $u(x)-x$.  The desired conclusion follows by Lemma \ref{lemma:samenewton}.\end{proof}

\begin{prop}\label{prop:countingroots}
Suppose $f,u\in\snc(\zp)$ is a minimal commuting pair. Let $\delta=1$ if $p>2$ and $\delta=2$ if $p=2$.  For all $n\geq \delta$, if $\pi\in\Oof(n+1)$, then $u^{\circ i}(\pi)=u^{\circ j}(\pi)$ if and only if $p^{n-\delta+1}\mid j-i$.
Moreover, $\Llf(n)=\Llu(n-\delta)$.
\end{prop}

\begin{proof}  Suppose $p>2$. By Lemma \ref{lemma:basecase}, $\Llf(1)=\Llu(0)$. 
Proceeding by induction, assume that $\Llf(n)=\Llu(n-1)$ for some $n\geq 1$. 
Let $\pi\in\Oof(n+1)$, and consider the series $f^{\circ n}(x)-f^{\circ n}(\pi)$, which has $p^n$ roots in $\ipalg$. Since $f^{\circ n}(\pi)\in\Llf(1)$, for any $i\geq 0$ we have $u^{\circ i}(f^{\circ n}(\pi))=f^{\circ n}(\pi)=f^{\circ n}(u^{\circ i}(\pi))$, so $u^{\circ i}(\pi)\in\Oof(n+1)$.  
Suppose that for some $i$ and $j$ we have $u^{\circ i}(\pi)=u^{\circ j}(\pi)$, and so $u^{\circ(j-i)}(\pi)=\pi$.  If $i\neq j$, write $j-i=rp^s$ with $p\nmid r$.  Applying Lemma \ref{lemma:samenewton} to $u^{\circ p^s}(x)-x$ and $u^{\circ rp^s}(x)-x$, we get $u^{\circ p^s}(\pi)=\pi$. This is impossible if $s\leq n-1$, since $\pi\notin\Llf(n)=\Llu(n-1)\supseteq\Llu(s)$.  A similar argument shows that $u^{\circ p^n}(\pi)=\pi$. Therefore, $\Llf(n+1)\subseteq\Llu(n)$.  By Lemma \ref{lemma:samenewton},  $\Llf(n+1)=\Llu(n)$, concluding the proof for odd primes.

Suppose $p=2$. The inductive step of this proof will be similar to the $p>2$ case, but we must first show that $\Llf(2)=\Llu(0)$ and $\Llf(3)=\Llu(1)$. Suppose $\alpha$ is the nonzero root of $f$.  Lemma \ref{lemma:totram} implies $v_2(\alpha)=1=v_2(u(\alpha))$. Since $f(u(\alpha))=0$, then $\alpha\in\Llu(0)$. Let $\beta,\beta'$ be the two roots of $f(x)-\alpha$. Then $v_2(\beta)=v_2(\beta')=1/2$. Also,  $\{u(\beta),u(\beta')\}=\{\beta,\beta'\}$, so that $\beta,\beta'\in\Llu(1)$. (We will later show that u fixes $\beta$ and $\beta'$.)  Next, let $\gamma_1,\gamma_2$ be the two roots of $f(x)-\beta$.  Thus $u^{\circ 2}(\beta)=\beta=u^{\circ 2}(f(\gamma_1))=f(u^{\circ 2}(\gamma_1))$, and so $u^{\circ 2}(\gamma_1)=\gamma_1$ or $\gamma_2$.  We show the impossibility of the latter case by considering $v_2(\gamma_2-\gamma_1)$. 

Recall that by Lemma \ref{lemma:basecase}, $\alpha$ is the only element of $\Llf=\Llu$ of $\q_2$-valuation $1$, $\beta_1$ and $\beta_2$ are the only elements of $\q_2$-valuation $1/2$, and all other roots have $\q_2$-valuation $1/2^k$ for $k\geq 2$. Therefore, $\Newt^-(u^{\circ 2}(x)-x)$ must have a segment of length $1$ and slope $-1$, and one segment of length $2$ and slope $-2$, and all other segments of slope $-1/2^k$ for $k\geq 2$.  The dotted line in Figure \ref{figure:newtonu2} corresponds to the smallest possible slope for the third segment of $\Newt^-(u^{\circ 2}(x)-x)$.

\setlength{\unitlength}{0.1cm}
\begin{figure}[h]
\begin{picture}(150,40)
\linethickness{0.25mm}
\put(20,0){\vector(0,1){40}}
\put(20,0){\vector(1,0){90}}
\put(30,30){\circle*{1}}
\put(40,20){\circle*{1}}
\put(60,10){\circle*{1}}
\linethickness{0.1mm}
\put(30,30){\line(1,-1){10}}
\put(40,20){\line(2,-1){20}}

\multiput(19,30)(0,-10){3}{\line(1,0){2}}
\put(16,29){$3$}
\put(16,19){$2$}
\put(16,9){$1$}

\put(30,-1){\line(0,1){2}}
\put(29,-4){$1$}

\put(40,-1){\line(0,1){2}}
\put(39,-4){$2$}

\put(60,-1){\line(0,1){2}}
\put(59,-4){$4$}

\put(100,-1){\line(0,1){2}}
\put(99,-4){$8$}

\multiput(60,10)(6,-1.5){7}{\line(4,-1){4}}

\end{picture}
\caption{$\Newt^-(u^{\circ 2}(x)-x)$}\label{figure:newtonu2}
\end{figure}

\noindent So if $u^{\circ 2}(x)-x=\sum_i b_ix^i$, then $v_2(b_1)=3$, $v_2(b_2)=2$, $v_2(b_3)\geq 2$, $v_2(b_4)=1$, and $v_2(b_i)\geq 1$ for $5\leq i\leq 7$. Thus $v_2(u^{\circ 2}(\gamma_1)-\gamma_1)\geq \min_i\{v_2(b_i\gamma_1^i)\}=\min_i\{v_2(b_i)+i/4\}\geq 2$.

But $\gamma_1$ and $\gamma_2$ are roots of an Eisenstein monic quadratic polynomial over $\zp[\beta]$, and so $(\gamma_2-\gamma_1)^2=b^2-4c$, where $v_2(b) \geq v_2(c)=1/2$.  Since $v_2(b^2)$ is an integer while $v_2(4c)=5/2$, the isosceles triangle principle guarantees that 
$$v_2(\gamma_2-\gamma_1)=\begin{cases} 1/2 & \mathrm{if}\; v_2(b)=1/2\\ 1 & \mathrm{if}\; v_2(b)=1\\ 5/4& \mathrm{if} \; v_2(b)\geq 3/2\end{cases}$$  
Therefore, $u^{\circ 2}(\gamma_1)$ cannot equal $\gamma_2$.  A similar proof shows that the roots of $f(x)-\beta'$ are fixed points of $u^{\circ 2}(x)$, so that $\Llf(3)\subseteq\Llu(1)$.  Lemma \ref{lemma:samenewton} implies $\Llf(3)=\Llu(1)$.  Finally, noting that $\Llu(0)\subset\Llu(1)$ and that $u$ must fix two points in addition to $0$ and $\alpha$, we conclude that $\beta,\beta'\in\Llu(0)$.

The inductive step for $p=2$ follows a proof similar to that of $p>2$.
Assume that $\Llf(n)=\Llu(n-2)$ for some $n\geq 3$. Let $\pi\in\Oof(n+1)$, and consider the series $f^{\circ n}(x)-f^{\circ n}(\pi)$.
Since $f^{\circ n}(\pi)\in\Llf(1)$, for any $i\geq 0$ we have $u^{\circ i}(f^{\circ n}(\pi))=f^{\circ n}(\pi)=f^{\circ n}(u^{\circ i}(\pi))$, so $u^{\circ i}(\pi)\in\Oof(n+1)$.  
Suppose that for some $i$ and $j$ we have $u^{\circ i}(\pi)=u^{\circ j}(\pi)$, and so $u^{\circ(j-i)}(\pi)=\pi$.  If $i\neq j$, write $j-i=r2^s$ with $2\nmid r$.  Applying Lemma \ref{lemma:samenewton} to $u^{\circ 2^s}(x)-x$ and $u^{\circ r2^s}(x)-x$, we get $u^{\circ 2^s}(\pi)=\pi$. This is impossible if $s\leq n-2$, since $\pi\notin\Llf(n)=\Llu(n-2)\supseteq\Llu(s)$.  A similar argument shows that $u^{\circ 2^{n-1}}(\pi)=\pi$. Therefore, $\Llf(n+1)\subseteq\Llu(n-1)$.  By Lemma \ref{lemma:samenewton},  $\Llf(n+1)=\Llu(n-1)$. \end{proof}

\begin{cor}
Suppose $f,u\in\snc(\zp)$ is a minimal commuting pair. Then 
$e(\bar u)=e=\begin{cases}p-1&\mathrm{if}\;p>2\\2&\mathrm{if}\;p=2\end{cases}$. In particular, $\Normsep_\fp(\cala_{\bar u})$ is an extension of a finite group of order dividing $e$ by $\cala_{\bar u}$.
\end{cor}

\section{Abelian extensions and torsion series} 
Consider the following notation: 
$$\begin{array}{rrrrrrl}
K_n&=&\qp(\Llf(n))&\quad&K&=&\cup_{n\geq 1}K_n\\
\G_n&=&\Gal(K_n/\qp)&\quad&\G&=&\displaystyle{\lim_{\leftarrow}\G_n}=\Gal(K/\qp)\end{array}$$  
For the remainder of the note, we will assume that for each $n$, $K_n$ is generated by any single element of $\Oof(n)$, or equivalently, $\qp(\pi)/\qp$ is Galois for all $\pi\in\Llf$.  In this section, we show that $\G$ is abelian, and that $\bar u$ commutes with a torsion series of order $e$.

Call a sequence $\{\pi_n\}_{n\geq 0}$ of elements in $\Llf$ {\sl $f$-consistent} if $\pi_0=0$, $\pi_1\neq 0$, and $f(\pi_{n+1})=\pi_n$ for all $n\geq 0$ (see \cite[Page 329]{Lubin:Nads}); in particular, $\pi_n\in\Oof(n)$ for all $n>0$.  By Lemma \ref{lemma:totram}, for $n>0$, $K_n/\qp$ is totally ramified of degree $(p-1)p^{n-1}$ and $|G_n|=(p-1)p^{n-1}$.

Fix an $f$-consistent sequence $\{\pi_n\}_{n\geq 0}$ and let $\ug_n\in \G_n$ be defined by $\ug_n(\pi_n)=u(\pi_n)$.  Since the coefficients of $u$ are fixed by $\ug_n$, we have $\ug_n^i(\pi)=u^{\circ i}(\pi)$ for all $i$. Clearly, $\ug_1$ is trivial, as is $\G_1$ if $p=2$.

\begin{lemma}\label{lemma:cyclicsubgroups} Suppose $f,u\in\snc(\zp)$ is a minimal commuting pair, and $\qp(\pi)/\qp$ is Galois for all $\pi\in\Llf$.
\begin{enumerate} 
\item If $p>2$, then the Sylow $p$-subgroup of $\G_n$ is cyclic and generated by $\ug_n$ for all $n\geq 1$.
\item If $p=2$, then $\G_n$ contains a cyclic subgroup of order $2^{n-2}$ generated by $\ug_n$ for all $n\geq 3$.
\end{enumerate}
In both cases,  $\ug_{n+1}|_{\G_n}=\ug_n$, and so $\ug=\displaystyle{\lim_{\leftarrow}\ug_n}$ generates a procyclic subgroup of $\G$ of index $2$. 
\end{lemma}

\begin{proof} The result follows immediately form Proposition \ref{prop:countingroots}.
If $p>2$, $\ug_n^i=1$ if and only if $p^{n-1}\mid i$. If $p=2$ and $n\geq 3$, $\ug_n^i=1$ if and only if $p^{n-2}\mid i$.  In both cases, $\ug_{n+1}(\pi_n)=\ug_{n+1}(f(\pi_{n+1}))=f(\ug_{n+1}(\pi_{n+1}))=f\circ u(\pi_{n+1})=u\circ f(\pi_{n+1})=\ug_n(\pi_n)$.
\end{proof}



\begin{lemma}\label{lemma:leftright}
Suppose $\anyg,\anyh\in\G$ and $g,h\in\snc(\zp)$ such that for some $\pi\in K$, $\anyg(\pi)=g(\pi)$ and $\anyh(\pi)=h(\pi)$. Then $\anyg\anyh(\pi)=h\circ g(\pi)$.
\end{lemma}

\begin{proof}
A direct computation yields the result: $\anyg\anyh(\pi)=\anyg(h(\pi))=h(\anyg(\pi))=h\circ g(\pi)$.
\end{proof}

\begin{lemma}\label{lemma:agreement}
Suppose $h_1,h_2\in\zp[\![x]\!]$ and $k$ is an integer such that $h_1(\pi_n)\equiv h_2(\pi_n)$ mod $\pi_n^k$ for infinitely many $n$.  Then $\bar h_1(x)\equiv \bar h_2(x)$ mod $x^k$.
\end{lemma}

\begin{proof}
Suppose $\bar h_1(x)\equiv \bar h_2(x)+\bar d x^m$ mod $x^{m+1}$ for some $d\in\zp^\times$ and $m>0$.  Pick $n$ large enough so that $v_p(\pi_n^{m+1})\leq 1$ and $h_1(\pi_n)\equiv h_2(\pi_n)$ mod $\pi_n^k$.  Then $h_1(\pi_n)\equiv h_2(\pi_n)+d\pi^m$ mod $\pi_n^{m+1}$, which implies $m\geq k$.
\end{proof}

\begin{remark}\label{remark:realization}{\rm 
Let $\pi\in\Oof(n)$. The Galoisness of $\qp(\pi)/\qp$ is equivalent to the following: for each $\pi'\in\Oof(n)$ there exists a $g_{\pi'}\in\gnc(\zp)$ such that $g_{\pi'}(\pi)=\pi'$.
As such, all the Galois automorphisms are ``locally analytic.''  On the other hand, the relation between $\ug$, $\ug_n$, and $u$ suggests that at least some of the Galois automorphisms are ``globally analytic''; in fact, our main result aims to show that they all are.  We take a step in that direction by partially extending the relation between $\ug$, $\ug_n$, and $u$ to other elements of the Galois group. For $\anyg\in \G$, write $\anyg(\pi_n)=\sum_{i=1}^\infty c_{i,n}\pi_n^i$, where the coefficients $c_{i,n}$ are Teichm\"uller representatives, and let $g_n(x)=\sum_{i=1}^\infty c_{i,n}x^i$.  Note that $\anyg^i(\pi_n)=g_n^{\circ i}(\pi_n)$ for all $i$. We will call the sequence $\{g_n\}$ the {\sl realization} of $\anyg$. Let $\Gamma=\{\sum_{i=1}^\infty c_ix^i\in\snc(\zp)\suchthat c_i^p=c_i\}$. The topology of $\zp[\![x]\!]$ induced by the additive $\z$-valued valuation $v_x$ is equivalent to the product topology of $\zp^\n$ when each copy of $\zp$ has the discrete topology. Tychonoff's theorem thus implies that $\Gamma$ is a compact subset of $\zp[\![x]\!]$. So $\{g_n\}$ must have an accumulation point $g\in\Gamma$.}
\end{remark}

\begin{lemma}\label{lemma:normalcylifts}
With the above notation, if $\{g_n\}$ is a realization of $\anyg\in\G$ with an accumulation point $g$, then $\bar g\in\Norm_\fp(\bar u)$.
\end{lemma}

\begin{proof}
By Lemma \ref{lemma:leftright}, $\ug\anyg(\pi_n)=g_n\circ u(\pi_n)$ and $\anyg\ug^t(\pi_n)= u^{\circ t}\circ g_n(\pi_n)$ for all $n$.
Let $\{g_{n_\ell}\}$ be a subsequence of $\{g_n\}$ which converges to $g$.  Given $k>0$, pick $l$ large enough such that if $\ell\geq l$ then $g(x)\equiv g_{n_\ell}(x)$ mod $x^k$.  Thus, if $\ell\geq l$, we have the following congruences mod $\pi_{n_\ell}^k$.
\begin{eqnarray*}
g\circ u(\pi_n)&\equiv&\ug\anyg(\pi_{n_\ell})\\
&=&\anyg\ug^t(\pi_{n_\ell})\\
&\equiv&u^{\circ t}\circ g_n(\pi_{n_\ell})
\end{eqnarray*}
Lemma \ref{lemma:agreement} implies $\bar u\circ \bar g\equiv\bar g\circ \bar u^{\circ t}$ mod $x^k$.  Since $k$ was arbitrary, our result follows.
\end{proof}

\begin{remark}{\rm Suppose $\anyg$ and $g$ are as above.  Then $\bar{g}=X_{K_1}(\anyg)$ only if the $f$-consistent sequence $\{\pi_n\}_{n\geq 1}$ is norm-consistent too---that is, if $N_{K_m/K_n}(\pi_m)=\pi_n$ whenever $m\geq n$.}
\end{remark}

\subsection{Proof that $\G$ is abelian if $p>2$}

 

\begin{lemma} Suppose $p>2$, $f,u\in\snc(\zp)$ is a minimal commuting pair, and $\qp(\pi)/\qp$ is Galois for all $\pi\in\Llf$. Then for all $n\geq 1$ there exists an automorphism $\zg_n\in \G_n$ of order $p-1$, and $\zg_{n+1}|_{\G_n}=\zg_n$.
\end{lemma}

\begin{proof}
Following Lemma \ref{lemma:fandzeta} and Corollary \ref{cor:G1cyclic}, let $\zeta_{p-1}$ be a primitive $p-1$ root of unity, and let $\zg_1$ be the generator of $\G_1$ given by $\zg(\pi_1)\equiv \zeta_{p-1}\pi_1$ mod $\pi_1^2$.  Proceeding by induction, suppose for some $n\geq 1$ there exists a $\zg_n\in \G_n$ of order $p-1$. Let $\hat\zg_{n+1}\in \G_{n+1}$ be any lifting of $\zg_n$, and let $\zg_{n+1}=\hat\zg_{n+1}^{p^n}$.
\end{proof}






\begin{lemma}\label{lemma:zetaaction}
Suppose $p>2$. Let $\rho_n=\zg_n(\pi_n)$.  Then $\{\rho_n\}$ is an $f$-consistent sequence, and $\rho_n\equiv \zeta_{p-1}\pi_n$ mod $\pi_n^2$.
\end{lemma}

\begin{proof}
The $f$-consistency can be directly verified: $f(\rho_{n+1})=f(\zg_{n+1}(\pi_{n+1}))=\zg_{n+1}(f(\pi_{n+1}))=\zg_n(\pi_n)=\rho_n$. 

Lemma \ref{lemma:fandzeta} provides the second part of the result for $n=1$, so we proceed by induction on $n$.  Suppose that for some $n\geq 1$,  $\rho_n\equiv \zeta_{p-1}\pi_n$ mod $\pi_n^2$. By Lemma \ref{lemma:totram}, $v_p(\pi_n^2)=2/(p^{n-1}(p-1))>v_p(\pi_{n+1}^{p+1})=(p+1)/(p^n(p-1))$, and so $\rho_n\equiv\zeta_{p-1}\pi_n$ mod $\pi_{n+1}^{p+1}$.
By our hypothesis on the commuting pair, $f(x)\equiv ax^p$ mod $(p, x^{p+1})$ for some $a\in\zp^\times$.  And by Lemma \ref{lemma:totram}, $v_p(\pi_{n+1}^{p+1})=v_p(\rho_{n+1}^{p+1})=(p+1)/(p^n(p-1))<1=v_p(p)$.  Thus, $\rho_n=f(\rho_{n+1})\equiv a\rho_{n+1}^p$ mod $\pi_{n+1}^{p+1}$. 
 Finally, $\zeta_{p-1}\pi_n=\zeta_{p-1}f(\pi_{n+1})\equiv a\pi_{n+1}^p$ mod $\pi_{n+1}^{p+1}$.  Therefore, $a\rho_{n+1}^p\equiv \zeta_{p-1}a\pi_{n+1}^p$ mod $\pi_{n+1}^{p+1}$, which implies our result.
\end{proof}

\begin{prop}
Suppose $p>2$ and $f,u\in\snc(\zp)$ is a minimal commuting pair.
If $\qp(\pi)/\qp$ is Galois for all $\pi\in\Llf$ then $\G\cong Z_{p-1}\times\zp$.  In particular, $\G$ is abelian.
\end{prop}

\begin{proof}
Let $\zg=\displaystyle{\lim_{\leftarrow}\zg_n}$ and $\ug=\displaystyle{\lim_{\leftarrow}\ug_n}$. Then $\left<\ug\right>\cong\zp$ is a normal subgroup of $\G$ since for each $n$, $\ug_n$ generates the unique Sylow $p$-subgroup of $\G_n$ by Lemma \ref{lemma:cyclicsubgroups}.  Therefore, $\zg\ug\zg^{-1}=\ug^t$ for some $t\in\zp$.  Moreover, $\zg\ug\zg^{-1}=\zg^p\ug\zg^{-p}=\ug^{t^p}$. So $t$ must be a $p-1$ root of unity. We will complete the proof by showing that $t\equiv 1$ mod $p$, and hence $t=1$.

Following the notation and result of Lemma \ref{lemma:zetaaction}, write $\rho_2=\zeta_{p-1}\pi_2+c_2\pi_2^2+c_3\pi_2^3+\cdots$ with $c_i\in\zp$.  Recall from the hypothesis on the commuting pair that $u(x)\equiv x+bx^p$ mod $(p,x^{p+1})$ for some $b\in \zp^\times$, and so $u^{\circ t}(x)\equiv x+tbx^p$ mod $(p,x^{p+1})$.  The following congruences are mod $\pi_2^{p+1}$.
\begin{eqnarray*}
\rho_2+b\zeta_{p-1}\pi_2^p&\equiv&\rho_2+b\rho_2^p\\
&\equiv&u(\rho_2)\\
&=&u(\zg(\pi_2)\\
&=&\zg(u(\pi_2))\\
&=&\zg\ug(\pi_2)\\
&=&\ug^t\zg(\pi_2)\\
&\equiv&\ug^t(\zeta_{p-1}\pi_2+c_2\pi_2^2+c_3\pi_2^3+\cdots)\\
&=&\zeta_{p-1}\ug^t(\pi_2)+c_2\ug^t(\pi_2)^2+c_3\ug^t(\pi_2)^3+\cdots\\
&=&\zeta_{p-1}u^{\circ t}(\pi_2)+c_2u^{\circ t}(\pi_2)^2+c_3u^{\circ t}(\pi_2)^3+\cdots\\
&\equiv&\zeta_{p-1}(\pi_2+tb\pi_2^p)+c_2\pi_2^2+c_3\pi_2^3+\cdots\\
&=&\rho_2+tb\zeta_{p-1}\pi_2^p.
\end{eqnarray*}
Therefore, $t\equiv 1$ mod $p$, concluding the proof.
\end{proof}

\begin{cor}
Suppose $p>2$. Then $\G_n\cong Z_{p-1}\times Z_{p^{n-1}}$.
\end{cor}

\subsection{Proof that $\G$ is abelian if $p=2$}
 
\begin{prop}
Suppose $p=2$ and $f,u\in\snc(\zp)$ is a minimal commuting pair.
If $\q_2(\pi)/\q_2$ is Galois for all $\pi\in\Llf$ then $\G\cong Z_2\times\z_2$.  In particular, $\G$ is abelian.
\end{prop}

\begin{proof}
Write $\bar f(x)=\varphi(x^p)$ where $\varphi(x)\in\snc(\fp)$.  Observe that $\varphi$ and $\bar u$ commute.  

If $\varphi\in\cala_{\bar u}$, then $\varphi=\bar u^{\circ t}$ for some $t\in\zp$. Therefore, $u^{\circ -t}\circ f$ has the following two properties: it is congruent to $x^p$ mod $p$, and its linear coefficient is a uniformizer in $\zp$.  In other words, $u^{\circ -t}\circ f$ is an endomorphism of a height-one formal $\zp$-module.  Since $f$ and $u$ commute with $u^{\circ -t}\circ f$, they must both be endomorphisms of the same formal group. Our result follows.

If $\varphi\notin\cala_{\bar u}$, then by Theorem \ref{thm:LMS}, $\Normsep_\fp(\cala_{\bar u})$ must be abelian.  By Lemma \ref{lemma:normalcylifts}, $\G$ must be abelian as well.

Let $\Oof(2)=\{\beta,\beta'\}$ and $\Oof(3)=\{\gamma_1,\gamma_2,\gamma'_1,\gamma'_2\}$ such that $f(\gamma_1)=f(\gamma_2)=\beta$ and $f(\gamma'_1)=f(\gamma'_2)=\beta'$. From Proposition \ref{prop:countingroots}, we know that $\ug_2(\beta)=u(\beta)=\beta'$, and so $\ug_3(\gamma_1)$ is a root of $f(x)-\beta'$; say $\ug_3(\gamma_1)=u(\gamma_1)=\gamma'_1$. Since $\gamma_1\in\Llu(1)$, $\ug_3^2(\gamma_1)=u^{\circ 2}(\gamma_1)=\gamma_1$.  Clearly, $\ug_3(\gamma_2)=\gamma'_2$.  Let $\zg_3\in \G_3$ be given by $\zg_3(\gamma_1)=\gamma'_2$. Then $\zg_3^2(\gamma_1)=\zg_3(\gamma'_2)=\gamma_1$ since $\ug_3(\gamma'_2)=\gamma_2$. Thus, $\G_3$ contains two elements of order $2$, and so it is not cyclic.
\end{proof}

\begin{cor}
Suppose $p=2$.  Then for $n\geq 2$, $\G_n\cong Z_2\times Z_{2^{n-2}}$.
\end{cor}

\begin{cor}
The group $\Normsep_\fp(\cala_{\bar u})$ is abelian.
\end{cor}

The $f$-consistent sequence defined in Lemma \ref{lemma:zetaaction} for $p>2$ has an analogue for $p=2$: let $\zg\in\G\cong Z_2\times Z_{2^{n-2}}$ be the generator of $Z_2$,  and define $\rho_n=\zg(\pi_n)$.
For each $n$, write the $\pi_n$-adic expansion of $\rho_n$ as
$\rho_n=\zg(\pi_n)=\sum_{i=1}^{\infty}c_{i,n}\pi_n^i$,
where $c_{1,n}$ is a fixed primitive $p-1$ root of unity $\zeta_{p-1}$, and $c_{i,n}^p=c_{i,n}$.  
Let $w_n(x)=\sum_{i=1}^\infty c_{i,n}x^i$.  Recall from Remark \ref{remark:realization} that $\{w_n\}$ is a realization of $\zg$ in the compact set $\Gamma$, and so it has a convergent subsequence $\{w_{n_\ell}\}$ and a corresponding accumulation point $w$.

\begin{lemma}\label{lemma:p2torsion}
Suppose $p=2$.  Then $w_n\equiv x+x^2$ mod $x^3$ for all $n$, and so $w(x)\equiv x+x^2$ mod $x^3$ as well.
\end{lemma}

\begin{proof}
This follows from the fact that $K_2\not\subset K^\zg$; see for instance \cite[Chapter IV]{Serre}.
\end{proof}


\subsection{Torsion series over $\fp$}
We end this section with the construction of a torsion power series over $\fp$ of order $e$ that commutes with $\bar f$ and $\bar u$.

\begin{prop}\label{prop:rbarcommutes}
Suppose $f,u\in\snc(\zp)$ is a minimal commuting pair.  Then $\bar w$ commutes with $\bar{u}$. Moreover, $\bar{w}$ is torsion of order $e=e(\bar u)$. 
\end{prop}

\begin{proof}For $k>0$, pick $l$ such that if $\ell>l$ then $v_p(\pi_{n_\ell}^{k+1})\leq1$ and $v_x(w-w_{n_\ell})\geq k+1$. 

By Lemma \ref{lemma:leftright}, $\zg\ug(\pi_n)=u\circ w_n(\pi_n)$ and $\ug\zg(\pi_n)=w_n\circ u(\pi_n)$ for all $n$.
Recalling that $G$ is abelian, we have the following congruences mod $\pi_{n_\ell}^{k+1}$ for $\ell>l$. 
\begin{eqnarray*}
u\circ w(\pi_{n_\ell})&\equiv& u\circ w_n(\pi_{n_\ell})\\
&=&\zg\ug(\pi_{n_\ell})\\
&=&\ug\zg(\pi_{n_\ell})\\
&=& w_n\circ u(\pi_{n_\ell})\\
&\equiv&w\circ u(\pi_{n_\ell})
\end{eqnarray*}
By Lemma \ref{lemma:agreement}, $\bar u\circ\bar w\equiv\bar w\circ \bar u$ mod $x^{k+1}$.  Since $k$ was arbitrary, it follows that $\bar u\circ\bar w=\bar w\circ \bar u$ .

If $\ell>l$, then $\pi_{n_\ell}=\zg^e(\pi_{n_\ell})=w_n^{\circ e}(\pi_{n_\ell})\equiv w^{\circ e}(\pi_{n_\ell})$ mod $\pi_{n_\ell}^{k+1}$.
By Lemma \ref{lemma:agreement}, $\bar w^{\circ e}(x)\equiv x$ mod $x^{k+1}$.  Since $k$ was arbitrary, we see that the order of $\bar w$ divides $e$.  If $p>2$, then the order is exactly $p-1$ because $w'(0)$ is a primitive $p-1$ root of unity.  If $p=2$, then the order of $\bar w$ is $2$ by Lemma \ref{lemma:p2torsion}.
\end{proof}








\begin{cor}
Suppose $\theta\in\snc(\fp)$ is a torsion series of order $e$ that commutes with $\bar u$.  If $\theta'(0)=\bar w'(0)$ then $\theta=\bar w$.
\end{cor}

Therefore, $\bar w$ is independent of the choice of $\{w_n\}$.

By \cite[Corollary 6.2.1]{Lubin:Nads}, the noninvertible half of a commuting pair must be of the form $\bar f(x)=\varphi(x^{p^d})$ for some $d>0$ and $\varphi\in\kp[\![x]\!]$ invertible; $d$ is called the {\sl radicial degree} of $f$.  Since our commuting pair is minimal, then $d=1$.

\begin{cor}\label{cor:barwandvarphicommute}
The power series $\bar w$ and $\varphi$ commute.
\end{cor}

If $p=2$, then $\bar w\in\Nott(\f_2)$.  Torsion elements of the Nottingham group are well understood and well behaved.  For instance, all torsion elements of $\Nott(\fp)$ have order $p^d$ for some $d$.  Moreover, since any pro-$p$ group can be embedded in $\Nott(\fp)$, then in fact there exists a torsion element of order $p^d$ for every $d$.  Two torsion elements of order $p$, $x+ax^\ell+\cdots$ and $x+bx^m+\cdots$, are conjugate over $\Nott(\fp)$ if and only if $a=b$ and $\ell=m$. The situation is slightly less tidy for $d>1$.  In fact, for some time, ``an element of order $p^2$ [was] still not known'' (\cite[page ]{Camina}).  In \cite{Lubin:Nott}, the results of \cite{Klopsch} are generalized via local-class-field-theoretic methods that associate to torsion elements of $\Nott(\fp)$ certain characters on $1+x\fp[\![x]\!]$, and explicit elements of any order are exhibited.

\begin{cor}\label{cor:Klopsch}
Suppose $p=2$.  Then $w(x)$ is conjugate over $\gnc(\fp)$ to $\sum_{i=1}^\infty x^i$.
\end{cor}

\begin{proof}
This follows directly from \cite[Theorem ]{Klopsch}.
\end{proof}



\section{Proof of main result}

Recall from Remark \ref{remark:proofplan} that if $\zeta_e$ is a primitive $e^\mathrm{th}$ root of unity, then $z(x)=\sum_{i=1}^\infty d_ix^i=[\zeta_e]_f(x)=L_f^{\circ -1}\left(\zeta_e L_f(x)\right)\in\gnc(\qp)$ is the unique $e$-torsion series with linear coefficient $d_1=\zeta_e$ that commutes with $f$ and $u$.   Let $z_k(x)=\sum_{i=1}^k d_ix^i$ and continue to write $f(x)=\sum_{i=1}^\infty a_ix^i$.  Clearly, $z_1(x)=\zeta_e x\in\gnc(\zp)$ and $z_1\circ f\equiv f\circ z_1$ mod $x^2$.  Moreover, if $p=2$ and $z_2(x)=-x+d_2x^2$, then $f\circ z_2\equiv z_2\circ f$ mod $x^3$ implies $d_2=2a_2/(a_1^2-a_1)\in\z_2^\times$.

The proof of our main result, that $z\in\gnc(\zp)$, will proceed inductively. Let $\delta=1$ if $p>2$ and $\delta=2$ if $p=2$.  Suppose for some $k\geq \delta$ that $z_k\in\gnc(\zp)$ and $z_k\circ f\equiv f\circ z_k$ mod $x^{k+1}$. For $z_{k+1}(x)=z_k(x)+d_{k+1}x^{k+1}$ to commute with $f(x)$ mod $x^{k+2}$, we must have 
$$d_{k+1}(a_1^{k+1}-a_1)x^{k+1}\equiv f\circ z_k(x)-z_k\circ f(x)\quad \mathrm{mod}\;x^{k+2}$$ Therefore, our proof will be completed once we show that $\bar f\circ \bar z_k(x)-\bar z_k\circ \bar f(x)\equiv 0$ mod $x^{k+2}$ in Proposition \ref{prop:lastone} below.

Recall that $\{\pi_n\}$ is a fixed $f$-consistent sequence, $\zg\in\G$ is an automorphism of order $e$, and $\rho_n=\zg(\pi_n)$ is $f$-consistent as well. The proof of the main result will rely on the relationship between the valuation of $\rho_n-g(\pi_n)$ for some $g\in\gnc(\zp)$ and the extent to which the series $g$ commutes with $f$. We explore that relation in more detail in the next three lemmas.

\begin{lemma}\label{lemma:rhoandpi}
Suppose for some $n>0$, $\rho_n\equiv g(\pi_n)+c_{m,n}\pi_n^m$ mod $\pi_n^{m+1}$, where $g\in\zp[\![x]\!]$ and $c_{m,n}\in\zp^\times$. Suppose further than $f\circ g\equiv g\circ f$ mod $x^{k+1}$.  If $m\neq p^{n-1}$ then $mp\geq \min\{v_{K_n}(\rho_{n-1}-g(\pi_{n-1})),k+1\}$.  The inequality is strict  if $v_{K_n}(a_p \pi_n^{mp})>v_{K_n}(a_1\pi_n^m)$.
\end{lemma}

\begin{proof}
Write $\rho_n=g(\pi_n)+c_{m,n}\pi_n^m+D\pi_n^{m+1}$ for some $D\in\zp[\![\pi_n]\!]$. Then
\begin{eqnarray*}
f(\rho_n)&=&\sum_{i=1}^\infty a_i(g(\pi_n)+c_{m,n}\pi_n^m+d\pi_n^{m+1})^i\\
&=&f(g(\pi_n))+a_1c_{m,n}\pi_n^m+a_p(c_{m,n}\pi_n^m)^p+D_1p\pi_n^{m+1}+D_2(\pi_n^{m+1})^p
\end{eqnarray*}
where $D_1,D_2\in\zp[\![\pi_n]\!]$.
If $m\neq p^{n-1}$ then $p^{n-1}(p-1)+m=v_{K_n}(a_1c_{m,n}\pi_n^m)\neq v_{K_n}(a_p(c_{m,n}\pi_n^m)^p)=mp$, and so by the isosceles triangle principle $v_{K_n}(f(\rho_n)-f(g(\pi_n)))=v_{K_n}(a_1c_{m,n}\pi_n^m+a_p(c_{m,n}\pi_n^m)^p)$.  But $f(\rho_n)-f(g(\pi_n))=\rho_{n-1}-g(\pi_{n-1})+M\pi_n^{k+1}$ for some $M\in\zp[\![\pi_n[\!]$.
Therefore,
\begin{eqnarray*}
mp&\geq&\min\{p^{n-1}(p-1)+m,mp\}\\
&=&v_{K_n}(a_1c_{m,n}\pi_n^m+a_p(c_{m,n}\pi_n^m)^p)\\
&=&v_{K_n}(f(\rho_n)-f(g(\pi_n)))\\
&=&v_{K_n}(\rho_{n-1}-g(\pi_{n-1})+M\pi_n^{k+1})\\
&\geq&\min\{v_{K_n}(\rho_{n-1}-g(\pi_{n-1})),k+1\}
\end{eqnarray*}
\end{proof}

\begin{lemma}\label{lemma:j}
Suppose $g\in\gnc(\zp)$ 
and $n\geq\delta$. For each $c\in\zp^\times$ there exists a unique integer $1\leq j\leq p-1$ depending on $\bar c$, $\overline{g'(0)}$, and $n$ such that $g\circ u^{\circ jp^{n-\delta}}(\pi_{n+1})-g(\pi_{n+1})\equiv c\pi_{n+1}^{p^n}$ mod $\pi_{n+1}^{p^n+1}$.
\end{lemma}

\begin{proof}
By Proposition \ref{prop:countingroots}, $\wideg(u^{\circ p^{n-\delta}}(x))=p^n$.  Write $u^{\circ p^{n-\delta}}(x)-x=\sum_{i=1}^\infty b_ix^i$.  Thus, $b_{p^n}\in\zp^\times$, and if $1\leq i<p^n$ then $v_{K_{n+1}}(b_i\pi_{n+1}^i)\geq p^n(p-1)+i>p^n$.
So $u^{\circ p^{n-\delta}}(\pi_{n+1})\equiv \pi_{n+1}+b_{p^n}\pi_{n+1}^{p^n}$ mod $\pi_{n+1}^{p^n+1}$.  A direct computation then shows $u^{\circ jp^{n-\delta}}(\pi_{n+1})\equiv \pi_{n+1}+jb_{p^n}\pi_{n+1}^{p^n}$ mod $\pi_{n+1}^{p^n+1}$. So the proof is complete by solving for $j$ in $jg'(0)b_{p^n}\equiv c$ mod $p$.
\end{proof}

\begin{lemma}\label{lemma:rhoandpiII}
Suppose $h\in\gnc(\zp)$ such that $h'(0)=\zeta_{p-1}$ if $p>2$ and $\bar h(x)\equiv x+x^2$ mod $x^3$ if $p=2$. If $f\circ h(x)-h\circ f(x)\equiv 0$ mod $x^{k+1}$, then there exists $\ell\in\zp$ such that  $v_{K_n}(\rho_n-h\circ u^{\circ \ell}(\pi_n))\geq(k+1)/p$ for all $n$.  The inequality is strict if $v_{K_n}(a_p \pi_n^{mp})>v_{K_n}(a_1\pi_n^m)$. Moreover, $v_{K_n}(\rho_n-h\circ u^{\circ \ell}(\pi_n))\neq p^{n-1}$.
\end{lemma}

\begin{proof}
We will construct a Cauchy sequence of integers $\{\ell_n\}$ for which $v_{K_s}(\rho_s-h\circ u^{\circ \ell_n}(\pi_s))\geq(k+1)/p$ whenever $s\leq n$.  

Let $\ell_1=0$. If $\rho_1=h(\pi_1)$ the result follows trivially. If not, write $\rho_1\equiv h(\pi_1)+c_{m,1}\pi_1^m$ mod $\pi_1^{m+1}$ for some $c_{m,1}\in\zp^\times$. Note that $m>1=p^0$ by Lemma \ref{lemma:fandzeta} and so $p-1+m<mp$. Therefore, by Lemma \ref{lemma:rhoandpi}, $mp>\min\{v_{K_1}(0),k+1\}= k+1$.

If $p=2$, let $\ell_2=0$ as well. The result follows trivially if $\rho_2=h(\pi_2)$.  Otherwise, write $\rho_2\equiv h(\pi_2)+c_{m,2}\pi_2^m$ mod $\pi_2^{m+1}$ for some $c_{m,2}\in\z_2^\times$.  Note that $m>2$ by Lemma \ref{lemma:p2torsion}, and so $2(2-1)+m<2m$.  Therefore, by Lemma \ref{lemma:rhoandpi}, $2m>\min\{v_{K_2}(\rho_1-h(\pi_1)),k+1\}= k+1$.

Suppose that for some $n\geq \delta+1$ there exists an integer $\ell_n$ for which $v_{K_s}(\rho_s-h\circ u^{\circ \ell_n}(\pi_s))\geq(k+1)/p$ whenever $s\leq n$.

If $\rho_{n+1}-h\circ u^{\circ \ell_n}(\pi_{n+1})\equiv c\pi_{n+1}^{p^n}$ mod $\pi_{n+1}^{p^n}$ for some $c\in\zp^\times$, then by Lemma \ref{lemma:j} there exists a $1\leq j_{n+1}\leq p-1$ such that $h\circ u^{\circ \ell_n}\circ u^{\circ j_{n+1} p^{n-\delta}}(\pi_{n+1})- h\circ u^{\circ \ell_n}(\pi_{n+1})\equiv c\pi_{n+1}^{p^n}$ mod $\pi_{n+1}^{p^n+1}$.  If on the other hand $v_{K_{n+1}}(\rho_{n+1}-h\circ u^{\circ \ell_n}(\pi_{n+1}))\neq p^n$, let $j_{n+1}=0$. Let $\ell_{n+1}=\ell_n+j_{n+1}p^{n-\delta}$.
If $\rho_{n+1}=h\circ u^{\circ \ell_{n+1}}(\pi_{n+1})$ then the result follows trivially. Otherwise, write $\rho_{n+1}-h\circ u^{\circ \ell_{n+1}}(\pi_{n+1})\equiv c_{m,n+1}\pi_{n+1}^m$ mod $\pi_{n+1}^m$ for some $c_{m,n+1}\in\zp^\times$. Observe that if $s< n+1$, then $u^{\circ p^{n-\delta}}(\pi_s)=\pi_s$ by Proposition \ref{prop:countingroots}, so that $u^{\circ \ell_{n+1}}(\pi_s)=u^{\circ \ell_{n}}(\pi_s)$.
The choice of $j_{n+1}$ guarantees that $m\neq p^n$.  Therefore, by Lemma \ref{lemma:rhoandpi}, $mp>\min\{v_{K_{n+1}}(\rho_n-h\circ u^{\circ \ell_{n+1}}(\pi_n)),k+1\}= \min\{v_{K_{n+1}}(\rho_n-h\circ u^{\circ \ell_n}(\pi_n)),k+1\}=k+1$.
\end{proof}

\begin{prop}\label{prop:lastone} Suppose $z_k\in\gnc(\zp)$ with $z_k'(0)=\zeta_e$ for some $k\geq \delta$.  If $p=2$, suppose that $\bar z_2(x)=x+x^2$. If $f\circ z_k(x)-z_k\circ f(x)\equiv 0$ mod $x^{k+1}$ then $\bar f\circ \bar z_k(x)-\bar z_k\circ \bar f(x)\equiv 0$ mod $x^{k+2}$.
\end{prop}

\begin{proof} Write $f\circ z_k(x)-z_k\circ f(x)\equiv dx^{k+1}$ mod $x^{k+2}$ for some $d\in\zp$.

If $p\nmid k+1$ then $d\in p\zp$ by \cite[Corollary 6.2.1]{Lubin:Nads}.

If $p\mid k+1$, let $m=(k+1)/p$.  
For each $n$, Lemma \ref{lemma:rhoandpiII} allows us to write $\rho_n\equiv z_k\circ u^{\circ \ell}(\pi_n)+c_n\pi_n^m$ mod $\pi_n^{m+1}$ for some $\ell\in\zp$, where $c_n\in\zp$ is a Teichm\"uller representative.  Moreover, the choice of $\ell$ guarantees that if $m= p^{n-1}$, then $c_n= 0$.

Let $N$ be the least positive integer that satisfies $v_{K_N}(a_1\pi_N^m)>v_{K_N}(a_p\pi_N^{mp})$; note that $N\geq 2$. By Lemma \ref{lemma:rhoandpiII} again, $c_{N-1}=0$.  Moreover, for all $n\geq N$, we have mod $\pi_n^{mp+1}$\\
$\begin{array}[t]{rcl}
\rho_{n-1}&\equiv&z_k\circ u^{\circ \ell}(\pi_{n-1})+c_{n-1}\pi_{n-1}^m\\          
&=&z_k\circ u^{\circ \ell}(f(\pi_n))+c_{n-1}f(\pi_n)^m\\
&\equiv&z_k\circ u^{\circ \ell}(f(\pi_n))+c_{n-1}(a_p\pi_n^p)^m
\end{array}
\begin{array}[t]{rcl}
\rho_{n-1}&=&f(\rho_n)\\
&\equiv&f(z_k\circ u^{\circ \ell}(\pi_n)+c_n\pi_n^m)\\
&\equiv&f(z_k\circ u^{\circ \ell}(\pi_n))+a_p(c_n\pi_n^m)^p
\end{array}$

Therefore, $f(z_k\circ u^{\circ \ell}(\pi_n))+a_p(c_n\pi_n^m)^p\equiv z_k\circ u^{\circ \ell}(f(\pi_n))+c_{n-1}(a_p\pi_n^p)^m$ mod $\pi_n^{pm+1}$, and so
$$c_n\equiv \frac{z_k\circ u^{\circ \ell}(f(\pi_n))-f( z_k\circ u^{\circ \ell}(\pi_n))}{a_p}+c_{n-1}a_p^{m-1} \quad \mathrm{mod}\;\pi_n$$
Iterating this last congruence, we get
$$c_n\equiv \frac{z_k\circ u^{\circ \ell}(f(\pi_n))-f(z_k\circ u^{\circ \ell}(\pi_n))}{a_p}\sum_{i=N}^{n}a_p^{(m-1)(n-i)}\quad \mathrm{mod}\;\pi_n$$
But $\sum_{i=N}^{n}a_p^{(m-1)(n-i)}\equiv 0$ mod $p$ for infinitely many $n$, and so too $c_n\equiv 0$ mod $p$ for infinitely many $n$.  Thus by Lemma \ref{lemma:agreement}, $\bar w\equiv\bar z_k\circ \bar u^{\circ \ell}$ mod $x^{m+1}$.

Now recall that $\bar w$ commutes with $\varphi$ by Corollary \ref{cor:barwandvarphicommute}. So $0=\varphi\circ \bar w(x)-\bar w\circ \varphi(x)\equiv \varphi\circ\bar z_k\circ\bar u^{\circ \ell}(x)-\bar z_k\circ\bar u^{\circ \ell}\circ\varphi(x)$ mod $x^{m+1}$. And since $\bar u$ commutes with $\bar w$ by Proposition \ref{prop:rbarcommutes}, we also have $\varphi\circ\bar z_k(x)-\bar z_k\circ\varphi(x)$ mod $x^{m+1}$. Therefore
$$0\equiv \varphi\circ\bar z_k(x^p)-\bar z_k\circ\varphi(x^p)\quad\mathrm{mod}\;x^{mp+1},$$
yielding our desired result since $mp=k+1$ and $\varphi(x^p)=\bar f(x)$.\end{proof}


\end{document}